\documentclass{birkjour}
 \newtheorem{thm}{Theorem}[section]
 
 \newtheorem{lem}[thm]{Lemma}
 \newtheorem{prop}[thm]{Proposition}
 \theoremstyle{definition}
 \newtheorem{defn}[thm]{Definition}
 \theoremstyle{remark}
 \newtheorem{rem}[thm]{Remark}
 
 \numberwithin{equation}{section}

 \usepackage{tkz-euclide}
\usepackage{tikz-cd}
\usepackage{url}

\newcommand{\arcangle}{%
  \mathord{<\mspace{-9mu}\mathrel{)}\mspace{2mu}}%
}

\usepackage{cite}
\usepackage{hyperref}
\hypersetup{
	colorlinks,
	linkcolor = {blue},
	citecolor = {blue},
	filecolor = {blue},
	urlcolor = {blue} 
} 

 \definecolor{DarkNavy}{RGB}{0,200,255}  
\definecolor{VibrantMagenta}{RGB}{255,0,127}  
\definecolor{LightGray}{RGB}{211,211,211}  
\definecolor{LightGray2}{RGB}{221,221,221}  
\definecolor{newGreen}{RGB}{85, 214, 69}  
\definecolor{LightBlue}{RGB}{173,225,255} 
\definecolor{lPurple}{RGB}{245, 190, 215}
\definecolor{uuuuuu}{RGB}{0, 0, 0}
\definecolor{neonPurple}{RGB}{177, 0, 253}

\begin{document}

%
%
%
%
%
%
%
%
%

\title[A Family of Eight-Point Conics]
 {A Family of Eight-Point Conics Associated with the Cyclic Quadrilateral}

\author{K. Chomicz}

\address{
Faculty of Mathematics and Computer Science, \\
Jagiellonian University, \\ ul. Łojasiewicza 6, 30-348
Kraków, Poland.}

\email{kazikchomicz@gmail.com}

\author{M. Płatek}
\address{Independent Researcher,\\
        Kraków, Poland.}

\email{milosz@platek.org}

\author{K. Smolira}
\address{
Faculty of Mathematics and Computer Science, \\
Jagiellonian University, \\ ul. Łojasiewicza 6, 30-348
Kraków, Poland.}

\email{konstanty.smolira@student.uj.edu.pl}

\author{D. Wyrzykowski}
\address{College of Inter-faculty Individual Studies in Mathematics and Natural Sciences,\\
University of Warsaw, \\
ul. Stefana Banacha 2C, 02-097 Warsaw, Poland. }
\email{d.wyrzykows2@student.uw.edu.pl}



\date{ }

\begin{abstract}
We consider the following configuration.
Let $ABCD$ be a cyclic quadrilateral with circumcenter $O$, and for each vertex $X$, let $H_X$ be the orthocenter of the triangle formed by the other three. Then $A,\;B,\;C,\;D,\;H_A,\;H_B,\;H_C,\;H_D$ all lie on a single conic. In this paper we study a certain generalization of this fact as follows. For an arbitrary point $P_D$ on the Euler line of $\triangle ABC$, we define corresponding points $P_A, P_B, P_C$ on the respective Euler lines such that the ratio $P_XH_X : P_XO$ is constant for all $X$. We show that the four vertices $A,B,C,D$ and the four isogonal conjugates $Q_A,\;Q_B\;,Q_C\;,Q_D$ of the points $P_X$ all lie on a single conic. This result is given distinct treatments, synthetic, projective, and algebraic. Furthermore, we situate the points $P_X$ within the list of triangle centers.
\end{abstract}

\maketitle
\section{Introduction}

Consider a triangle $\triangle ABC$ with orthocenter $H$. It is well established that if one considers a rectangular hyperbola $\Phi$ passing through the vertices of $\triangle ABC$, then $H$ lies on $\Phi$ \cite{conics, yiu}. From this fact, it immediately follows that if one considers $D$ as the point of intersection of $\Phi$ with the circumcircle of $\triangle ABC$, the orthocenters of $\triangle ABD, \triangle ACD, \triangle BCD$ lie on $\Phi$ \cite{tienhoven}. We state this below.

\begin{prop}\label{orthocenters}
    Let $ABCD$ be a cyclic quadrilateral, and by $H_A$ denote the orthocenter of triangle $\triangle BCD$, similarly define $H_B, H_C, H_D$. Then, $A, B, C, D, \\ H_A, H_B, H_C, H_D$ all lie on a single conic. 
\end{prop}

This paper studies a generalization of this fact, extending it to a family of points defined on the Euler line. This generalization was first stated in \cite{wyrzyk}; however the proof offered was found to contain a fatal error, leading to the paper's retraction. This paper provides three distinct geometric treatments and establishes a connection to the modern theory of triangle centers \cite{kimberling, yiu, shinagawa}.

\textbf{Notation.} We work in the projective plane, denoting the line at infinity as $\ell_{\infty}$ and points on this line as $\infty_{X}$ or $\infty_k$ if we are refering to the infinity point of a line $k$. Moreover, for a given segment $AB$: $M_{AB}$ denotes its midpoint,  $|AB|$ denotes its Euclidean length. Furthermore, by $X(Y) \longmapsto X'(Y')$ we will denote a projectivity from $Y$ to $Y^{\prime}$ in which $X \in Y$ maps to $X^{\prime} \in Y'$, by $\mathcal{L}(X)$ we denote the set of lines passing through point $X$.

\section{Main Result}

\subsection{Preliminaries}

We recall a list of Lemmas.

\begin{lem}[Steiner's Theorem]\label{th2}
Given two pencils of lines, $P$ and $Q$, let $\phi$ be a projective map between them. Then:

\begin{enumerate}
\item If $\phi(PQ) = PQ$, the locus of points $x \cap \phi(x)$ for $x \in P$ forms a straight line $\Gamma$.
\item If $\phi(PQ) \neq PQ$, the locus of points $x \cap \phi(x)$ for $x \in P$ forms a conic $\Gamma$ that passes through points $P$ and $Q$.
\end{enumerate}
Moreover, the transformation
\begin{equation}
x(P) \longmapsto x \cap \phi(x)\;(\Gamma)
\end{equation}
is projective.
\end{lem}

For the proof of Lemma \ref{th2} and related theory see \cite{coxeter}. A direct  corollary of Steiner's Theorem is the following. 

\begin{lem}[Isogonal Conjugate of a Line]\label{th3}
Given a triangle $ABC$ and a line $\ell$, the locus of the isogonal conjugates of points on $\ell$ with respect to $\triangle ABC$ is:

\begin{enumerate}
\item a line $\Gamma$, if $\ell$ passes through any of the vertices of $\triangle ABC$;
\item a circumconic $\Gamma$ of $\triangle ABC$, otherwise.
\end{enumerate}
Moreover, the transformation
\begin{equation}
X(\ell) \longmapsto X'(\Gamma)
\end{equation}
is projective, where $X \in \ell$ and $X'$ is the isogonal conjugate of the point $X$ with respect to triangle $ABC$.

\end{lem}

\begin{rem}
    In the case of the Euler line, this circumconic is the Je\v{r}\'abek hyperbola \cite{kimberling}, whose center is $X_{125}$ in the ETC. 
\end{rem}

\begin{lem}\label{th4}
Let $\mathcal{H}$ be a rectangular hyperbola, and let $A,B,C$ be distinct points on it, and let $H$ denote the orthocenter of $\triangle ABC$. If $\mathcal{H}$ intersects the circumcircle of $\triangle ABC$ for the fourth time at $D$, then the midpoint of $HD$ is the center of $\mathcal{H}$.
\end{lem}

\begin{proof}
    Let $O$ be the circumcenter of $\triangle ABC$. By $\infty_D$ denote the isogonal conjugate of $D$ with respect to $\triangle ABC$. Let the line $O\infty_D$ intersect the circumcircle of $\triangle ABC$ at $K$ and $L$, and by $\infty_K$ and $\infty_L$ we denote the isogonal conjugates of $K$ and $L$ with respect to $\triangle ABC$ respectively. Then, we have
    \begin{equation}
        -1 = (O,\infty_D;K,L) = (H,D;\infty_K,\infty_L),
    \end{equation}
    which implies that the asymptotes of $\mathcal{H}$ intersect each other on $HD$. By symmetry, the intersection point must be in the midpoint--which is the center of $\mathcal{H}$.
\end{proof}

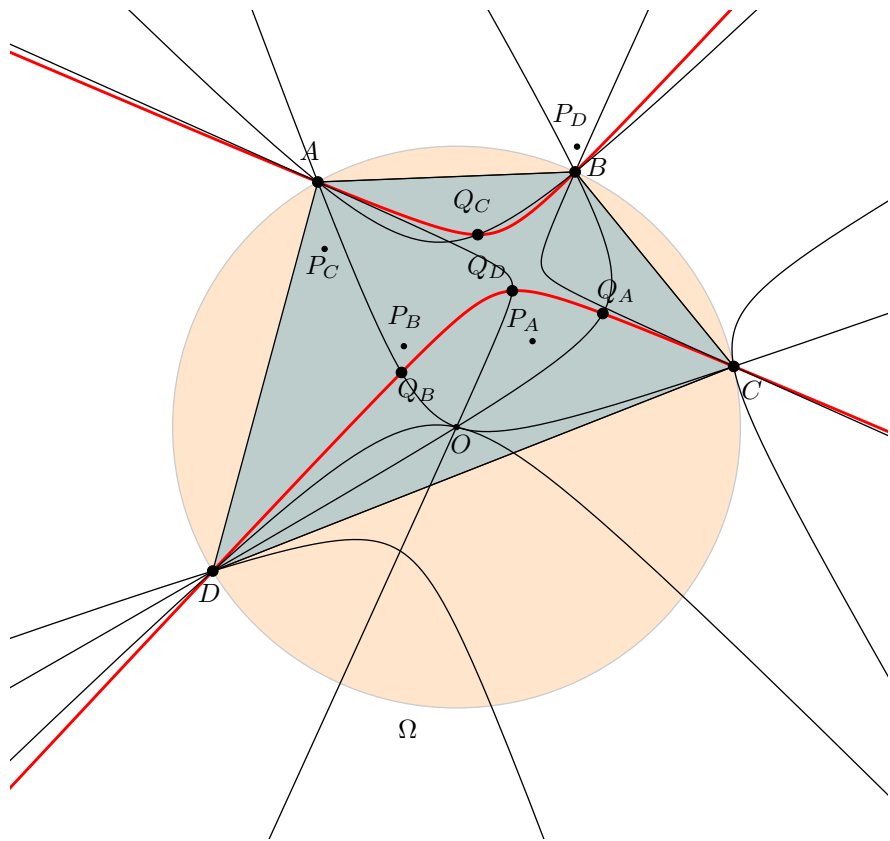
\begin{figure}[ht]
    \centering
    \begin{tikzpicture}[line cap=round,line join=round,>=triangle 45,x=1cm,y=1cm]
\clip(-5.884400429503071,-5.479897695097389) rectangle (5.686422189363454,5.548923238534636);
\draw [fill = LightBlue, opacity = 1] (-1.82782099627875,3.2648231813625923) -- (1.5635872464765799,3.3992932975334424) -- (3.6533450748449545,0.8081273192425275) -- (-3.212662118153669,-1.9180203634425737);
\draw [line width = 0.5pt, fill = orange, opacity = 0.2] (0,0) circle (3.7416573867739413cm);
\draw [line width = 0.5pt] (-1.82782099627875,3.2648231813625923)-- (-3.212662118153669,-1.9180203634425737);
\draw [line width = 0.5pt] (-3.212662118153669,-1.9180203634425737)-- (3.6533450748449545,0.8081273192425275);
\draw [line width = 0.5pt] (3.6533450748449545,0.8081273192425275)-- (1.5635872464765799,3.3992932975334424);
\draw [line width = 0.5pt] (1.5635872464765799,3.3992932975334424)-- (-1.82782099627875,3.2648231813625923);
\draw [samples=50,domain=-0.99:0.99,rotate around={-78.03692482797108:(0.5449698876154211,2.1903525319212744)},xshift=0.5449698876154211cm,yshift=2.1903525319212744cm,line width = 1.1pt, red] plot ({0.3993039015237368*(1+(\x)^2)/(1-(\x)^2)},{0.562232466212539*2*(\x)/(1-(\x)^2)});
\draw [samples=50,domain=-0.99:0.99,rotate around={-78.03692482797108:(0.5449698876154211,2.1903525319212744)},xshift=0.5449698876154211cm,yshift=2.1903525319212744cm,line width = 1.1pt, red] plot ({0.3993039015237368*(-1-(\x)^2)/(1-(\x)^2)},{0.562232466212539*(-2)*(\x)/(1-(\x)^2)});
\draw [samples=50,domain=-0.99:0.99,rotate around={-90.42327901738525:(-0.24576009122841694,1.2452701522383154)},xshift=-0.24576009122841694cm,yshift=1.2452701522383154cm,line width = 0.5pt] plot ({1.2170019876873372*(1+(\x)^2)/(1-(\x)^2)},{1.2170019876873372*2*(\x)/(1-(\x)^2)});
\draw [samples=50,domain=-0.99:0.99,rotate around={-90.42327901738525:(-0.24576009122841694,1.2452701522383154)},xshift=-0.24576009122841694cm,yshift=1.2452701522383154cm,line width = 0.5pt] plot ({1.2170019876873372*(-1-(\x)^2)/(1-(\x)^2)},{1.2170019876873372*(-2)*(\x)/(1-(\x)^2)});
\draw [samples=50,domain=-0.99:0.99,rotate around={64.89197434815925:(-0.4651061266088616,-0.7793614463192434)},xshift=-0.4651061266088616cm,yshift=-0.7793614463192434cm,line width = 0.5pt] plot ({0.8985343439661689*(1+(\x)^2)/(1-(\x)^2)},{0.8985343439661689*2*(\x)/(1-(\x)^2)});
\draw [samples=50,domain=-0.99:0.99,rotate around={64.89197434815925:(-0.4651061266088616,-0.7793614463192434)},xshift=-0.4651061266088616cm,yshift=-0.7793614463192434cm,line width = 0.5pt] plot ({0.8985343439661689*(-1-(\x)^2)/(1-(\x)^2)},{0.8985343439661689*(-2)*(\x)/(1-(\x)^2)});
\draw [samples=50,domain=-0.99:0.99,rotate around={-15.094689508132747:(2.8385206230761,1.5020350066517598)},xshift=2.8385206230761cm,yshift=1.5020350066517598cm,line width = 0.5pt] plot ({0.8522525069311683*(1+(\x)^2)/(1-(\x)^2)},{0.8522525069311683*2*(\x)/(1-(\x)^2)});
\draw [samples=50,domain=-0.99:0.99,rotate around={-15.094689508132747:(2.8385206230761,1.5020350066517598)},xshift=2.8385206230761cm,yshift=1.5020350066517598cm,line width = 0.5pt] plot ({0.8522525069311683*(-1-(\x)^2)/(1-(\x)^2)},{0.8522525069311683*(-2)*(\x)/(1-(\x)^2)});
\draw [samples=50,domain=-0.99:0.99,rotate around={20.705721045816247:(0.9279156724785139,2.0295864831850823)},xshift=0.9279156724785139cm,yshift=2.0295864831850823cm,line width = 0.5pt] plot ({0.2187803928772484*(1+(\x)^2)/(1-(\x)^2)},{0.2187803928772484*2*(\x)/(1-(\x)^2)});
\draw [samples=50,domain=-0.99:0.99,rotate around={20.705721045816247:(0.9279156724785139,2.0295864831850823)},xshift=0.9279156724785139cm,yshift=2.0295864831850823cm,line width = 0.5pt] plot ({0.2187803928772484*(-1-(\x)^2)/(1-(\x)^2)},{0.2187803928772484*(-2)*(\x)/(1-(\x)^2)});
\draw [fill=black] (-1.82782099627875,3.2648231813625923) circle (2pt);
\draw[color=black] (-1.939488272862,3.6806909771070444) node {$A$};
\draw [fill=black] (1.5635872464765799,3.3992932975334424) circle (2pt);
\draw[color=black] (1.852537060741967,3.4584477341119406) node {$B$};
\draw [fill=black] (3.6533450748449545,0.8081273192425275) circle (2pt);
\draw[color=black, yshift = -3] (3.8943968557594872,0.59706598054998) node {$C$};
\draw [fill=black] (-3.212662118153669,-1.9180203634425737) circle (2pt);
\draw[color=black] (-3.2590575281454317,-2.208754962263205) node {$D$};
\draw [fill=black] (1.9281291940605165,1.5129996127975716) circle (2pt);
\draw[color=black, yshift = 14] (2.1025607091114593,1.3193565202840674) node {$Q_{A}$};
\draw [fill=black] (1.0021351015839333,1.144700126666698) circle (1pt);
\draw[color=black, xshift = 4, yshift = -2] (0.7413208457664454,1.499929155217589) node {$P_{A}$};
\draw [fill=black] (-0.6935690197937322,1.0774650685812728) circle (1pt);
\draw[color=black] (-0.6754798283273444,1.472148749843201) node {$P_{B}$};
\draw [fill=black] (-1.7384479339779197,2.3730480577267303) circle (1pt);
\draw[color=black,yshift = -4, xshift = 4] (-1.8978176648004181,2.277780505700452) node {$P_{C}$};
\draw [fill=black] (-0.7250109899920796,0.7272038668785127) circle (2pt);
\draw[color=black, xshift = 8] (-0.8004916525120905,0.48594435905242817) node {$Q_{B}$};
\draw [fill=black] (0.2797844924383188,2.5632167839852356) circle (2pt);
\draw[color=black] (0.21349314365307276,3.0000710454345394) node {$Q_{C}$};
\draw [fill=black] (0.7354851953127071,1.8150093458969967) circle (2pt);
\draw[color=black] (0.407955981273789,2.1249882761413184) node {$Q_{D}$};
\draw [fill=black, xshift = -3] (1.694555662521393,3.7361218990692833) circle (1pt);
\draw[color=black] (1.5052819935621162,4.15295786847164) node {$P_{D}$};
\draw [fill=uuuuuu] (0,0) circle (1pt);
\draw[color=uuuuuu, yshift = -2] (0.06070091409393856,-0.1530049645584952) node {$O$};
\draw[color=uuuuuu, yshift = -110, xshift = -20] (0.06070091409393856,-0.1530049645584952) node {$\Omega$};
\end{tikzpicture}
    \caption{$A, B, C, D, Q_{A}, Q_{B}, Q_{C}, Q_D$ lie on a conic.}
    \label{fig:figureMainLemma}
\end{figure}

Now we state the main result of this paper (Fig. \ref{fig:figureMainLemma}).

\begin{thm}
	\label{themG}
	Let $A, B, C, D$ be distinct ordered points lying on a circle, denoted by $\Omega$ with center \(O\). By $H_{D}$, $H_{B}$, $H_{C}$, $H_{A}$ denote the orthocenters of \(\triangle ABC, \triangle ACD, \) \(\triangle ABD, \triangle BCD\), and let \(P_{D}\) be any point on the Euler line of \(\triangle ABC\). Let \(P_{A}\) be the point on the Euler line of \(\triangle BCD\) such that
	\begin{equation}\label{eq2}
		\frac{P_{A}H_{A}}{P_{A}O} \;=\; \frac{P_{D}H_{D}}{P_{D}O},
	\end{equation}
	where \(H_{A}\) is the orthocenter of \(\triangle BCD\). Define \(P_{B},P_{C}\) similarly on the Euler lines of \(\triangle ACD\) and \(\triangle ABD\). Let \(Q_{D}\) be the isogonal conjugate of \(P_{D}\) with respect to \(\triangle ABC\), and define \(Q_{A},Q_{B},Q_{C}\) analogously in their respective triangles. Then the points
	\begin{equation}
		A,\;B,\;C,\;D,\;Q_{A},\;Q_{B},\;Q_{C},\;Q_{D}
	\end{equation}
	all lie on a single conic.
	\end{thm}

Without loss of generality, we will prove that the points $A, B, C, D, Q_A,$ and $Q_D$ lie on a single conic. By symmetry, an analogous argument can be constructed for the remaining pairs of $Q$ points. Since five points, no three collinear, uniquely determine a conic, this is sufficient to establish that all eight points lie on a single conic.

Furthermore, we assume that no sub-triangle is equilateral. For an equilateral triangle, the Euler line degenerates to a single point where the orthocenter and circumcenter coincide. In such a case, Theorem \ref{themG} reduces to that of Proposition \ref{orthocenters}. 

\subsection{More (or less) synthetic treatments}
    We give the first proof, using projective and synthetic tools. 
\begin{rem}[Degenerated cases]
    In all the proofs presented in this paper, whenever we refer to Lemma~\ref{th3} and Lemma~\ref{th4}, we assume that the locus of points described in these lemmas is a conic. If the locus happens to be a line, such cases can be handled separately by an argument of continuity in geometry, since degenerate cases are finitely many.
\end{rem}

    \begin{proof}
    
    By $\infty_A$ and $\infty_D$ denote the isogonal conjugates of points $A$ and $D$ with respect to triangles $DBC$ and $ABC$, respectively.  Since the quadrilateral $ABCD$ is cyclic, the isogonal conjugates $\infty_A$ and $\infty_D$ lie on the line at infinity.

     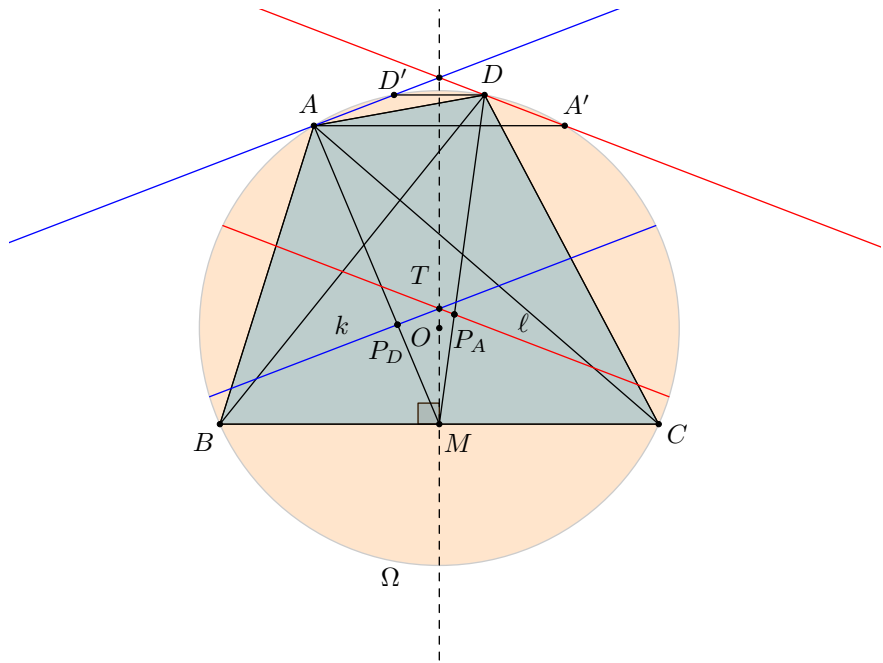
\begin{figure}[ht]
            \centering
            \begin{tikzpicture}[line cap=round,line join=round,>=triangle 45,x=1cm,y=1cm]
\clip(-5.670396082614365,-4.422141654689727) rectangle (5.820489864438509,4.242159189072243);

\draw [fill = LightBlue, opacity = 1] (2.8916431315084505,-1.28) -- (-2.891643131508451,-1.28) -- (-1.6528981964245268,2.6959094109885346) -- (0.5964694176999945,3.105515131785487) -- cycle;

\draw[line width=0.5pt,fill=black,fill opacity=0.1] (0,-0.9998207676433972) -- (-0.28017923235660286,-0.9998207676433971) -- (-0.28017923235660286,-1.28) -- (0,-1.28) -- cycle; 

\draw [line width = 0.5pt, fill =orange, opacity = 0.2, draw = black] (0,0) circle (3.1622776601683795cm);

\draw [line width = 0.5pt] (-1.6528981964245268,2.6959094109885346)-- (-2.891643131508451,-1.28);
\draw [line width = 0.5pt] (0.5964694176999945,3.105515131785487)-- (2.8916431315084505,-1.28);
\draw [line width = 0.5pt] (0.5964694176999945,3.105515131785487)-- (-2.891643131508451,-1.28);
\draw [line width = 0.5pt] (-1.6528981964245268,2.6959094109885346)-- (2.8916431315084505,-1.28);
\draw [line width = 0.5pt] (-1.6528981964245268,2.6959094109885346)-- (0.5964694176999945,3.105515131785487);
\draw [line width = 0.5pt] (-2.891643131508451,-1.28)-- (2.8916431315084505,-1.28);
\draw [line width = 0.5pt] (-1.6528981964245268,2.6959094109885346)-- (0,-1.28);
\draw [line width = 0.5pt] (0,-1.28)-- (0.5964694176999945,3.105515131785487);

\draw [line width = 0.5pt,domain=-9.670396082614365:8.820489864438509, blue] plot(\x,{(-7.113124955652221-0.82652949420379*\x)/-2.131731808927441});
\draw [line width = 0.5pt,domain=-9.670396082614365:8.820489864438509, red] plot(\x,{(-7.113124955652221--0.8265294942037908*\x)/-2.1317318089274404});

\draw [line width = 0.5pt, dashed] (0,-4.422141654689727) -- (0,4.242159189072243);
\draw [line width = 0.5pt] (-1.6528981964245268,2.6959094109885346)-- (1.6528981964245253,2.695909410988535);
\draw [line width = 0.5pt] (-0.5964694176999948,3.1055151317854865)-- (0.5964694176999945,3.105515131785487);
\draw [line width = 0.5pt, blue] (-3.027082141997325,-0.914753357797)-- (2.8525358494462987,1.3649319498142323);
\draw [line width = 0.5pt, red] (-2.8525358494462982,1.3649319498142334)-- (3.027082141997325,-0.9147533577970015);

\draw [fill=black] (2.8916431315084505,-1.28) circle (1pt);
\draw[color=black, yshift = -12, xshift = 4] (2.9958607911168547,-0.9947238666466917) node {$C$};
\draw [fill=black] (-2.891643131508451,-1.28) circle (1pt);
\draw[color=black, yshift = -15, xshift = -9] (-2.7891449551754017,-0.9947238666466917) node {$B$};
\draw [fill=black] (-1.6528981964245268,2.6959094109885346) circle (1pt);
\draw[color=black, xshift = -5] (-1.547614041587566,2.9808166119697006) node {$A$};
\draw [fill=black] (0.5964694176999945,3.105515131785487) circle (1pt);
\draw[color=black] (0.6977078234117116,3.3902576579401598) node {$D$};
\draw [fill=black] (0,-1.28) circle (1pt);
\draw[color=black, xshift = 3, yshift = 4] (0.14298124500012535,-1.668320426146479) node {$M$};
\draw [fill=black] (0.19882313923333134,0.1818383772618291) circle (1pt);
\draw [fill=black] (-0.5509660654748424,0.045303136996178384) circle (1pt);
\draw [fill=black] (0,0) circle (1pt);
\draw[color=black] (-0.24004424961739848,-0.12301067199990795) node {$O$};
\draw [fill=black] (-0.5509660654748424,0.045303136996178384) circle (1pt);
\draw[color=black, xshift = 5] (-0.8740174820877827,-0.3211273071469042) node {$P_D$};
\draw [fill=black] (0.1988231392333313,0.18183837726182908) circle (1pt);
\draw[color=black, yshift = -7] (0.4071367585294522,0.07510596314708834) node {$P_A$};
\draw[color=black, xshift = 100, yshift = 15] (-8.67981290687939,0.5637936631763458) node {};
\draw [fill=black] (-0.5964694176999948,3.1055151317854865) circle (1pt);
\draw[color=black, xshift = -5, yshift = -2] (-0.4381608847643936,3.3902576579401598) node {$D^{\prime}$};
\draw[color=black, xshift = -100, yshift = 30] (8.728035434703246,-0.21546510173517291) node {};
\draw [fill=black] (1.6528981964245253,2.695909410988535) circle (1pt);
\draw[color=black] (1.807160980234884,2.9808166119697006) node {$A^{\prime}$};
\draw [fill=black] (0,0.2589274313437391) circle (1pt);
\draw[color=black] (-0.2532520252938648,0.7090791956174765) node {$T$};
\draw [fill=black] (0,3.3367822940312175) circle (1pt);
\draw[color=black, xshift = -20, yshift = -9] (-0.5834464172055234,0.36567702802934954) node {$k$};
\draw[color=black, yshift = -2] (1.1071488693821683,0.14114484152942042) node {$\ell$};
\draw[color=uuuuuu, yshift = -90, xshift = -20] (0.06070091409393856,-0.1530049645584952) node {$\Omega$};
\end{tikzpicture}
            \caption{Point $T$ lies on the perpendicular bisector of the segment $BC$.}
            \label{platek2}
        \end{figure}

    Let line $k$ be the line passing through the points $P_D$ and $\infty_D$, and let line $\ell$ be the line passing through the points $P_A$ and $\infty_A$. Denote $T$ as the intersection of $k$ and $\ell$. Note that since $\infty_A$ and $\infty_D$ are fixed and independent of the specific choice of $P_A$ and $P_D$, and $P_AP_D \parallel H_AH_D$, the locus of the point $T$ is a fixed straight line passing through $O$. We will show that this line is the perpendicular bisector of the side $BC$--denoted by $m$ (Fig. \ref{platek2}).
        
    Now consider the case $P_A = H_A$, which implies $P_D = H_D$. It suffices to demonstrate that $T$ lies on the perpendicular bisector of the side $BC$. Consider $A'$, $D'$, $H_A'$ and $H_D'$ -- reflections of $A$, $D$, $H_A$ and $H_D$, respectively, in the perpendicular bisector of $BC$. We have $AD' \parallel k$ and $DA' \parallel \ell$ from the definition and that $AD'$ and $DA'$ are reflections of each other in our bisector. Moreover, 
     \begin{equation}
         |AH_D| = 2|OM_{BC}| =  |DH_A| = |D'H_A'|.
     \end{equation}
    Now, from $AH_D \parallel D'H_A'$, we find that $AD'H_DH_A'$ is a parallelogram, which gives $AD' \parallel H_DH_A' \implies \ell = H_DH_A'$. Similarly we get that $k = H_AH_D'$, and our claim follows.

        Let $T_1$ and $T_2$ denote the isogonal conjugates of the point $T$ with respect to $\triangle ABC$ and $\triangle DBC$, respectively. By Lemma \ref{th3}, the line $\ell$, under isogonal conjugation with respect to $\triangle DBC$, maps to a circumconic of $\triangle DBC$. Since the point $\infty_A$ lies on $\ell$ and is the isogonal conjugate of point $A$ with respect to $\triangle DBC$, it follows that $A$ also lies on this conic. Furthermore, by the definition of point $T_2$ and the fact that point $P_A$ lies on $\ell$, both $T_2$ and $Q_A$ lie on this conic.

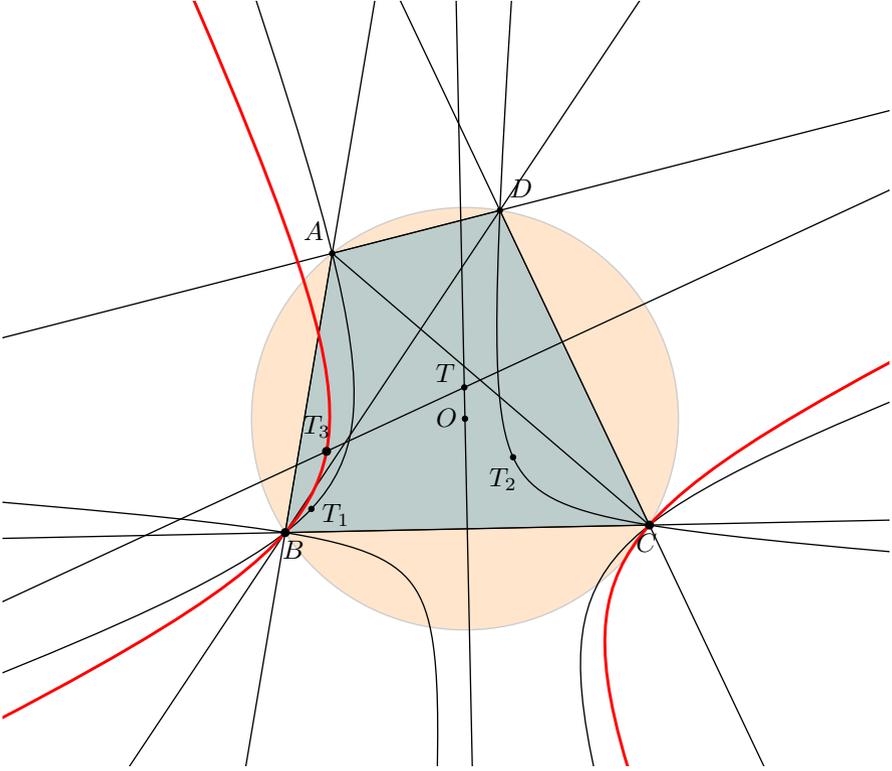
\begin{figure}
    \centering
    \begin{tikzpicture}[line cap=round,line join=round,>=triangle 45,x=1cm,y=1cm]
\clip(-8.325829998854992,-4.74197886113063) rectangle (3.364513866782811,5.4333056492741445);
\draw [fill = LightBlue, opacity = 1] (-3.98,2.08) -- (-4.6,-1.64) -- (0.2,-1.54) -- (-1.7701037227306569,2.652151402568374) -- cycle;
\draw [line width = 0.5pt, fill = orange, opacity = 0.2] (-2.2305574912891983,-0.12324041811846673) circle (2.813328532417638cm);
\draw [line width = 0.5pt,domain=-8.325829998854992:7.364513866782811] plot(\x,{(-16.0952-3.72*\x)/-0.62});
\draw [line width = 0.5pt,domain=-8.325829998854992:7.364513866782811] plot(\x,{(-7.412--0.1*\x)/4.8});
\draw [line width = 0.5pt,domain=-8.325829998854992:7.364513866782811] plot(\x,{(--2.1955294524915367--4.192151402568374*\x)/-1.9701037227306568});
\draw [line width = 0.5pt,domain=-8.325829998854992:7.364513866782811] plot(\x,{(-6.8737468389423615-0.572151402568374*\x)/-2.209896277269343});
\draw [line width = 0.5pt,domain=-8.325829998854992:7.364513866782811] plot(\x,{(--10.719--4.8*\x)/-0.1});
\draw [line width = 0.5pt] (-3.98,2.08)-- (0.2,-1.54);
\draw [line width = 0.5pt,domain=-8.325829998854992:7.364513866782811] plot(\x,{(-15.102866557092797-4.292151402568374*\x)/-2.829896277269343});

\draw [samples=50,domain=-0.99:0.99,rotate around={-25.177912521427697:(-2.2,-1.59)},xshift=-2.2cm,yshift=-1.59cm,line width = 0.5pt] plot ({1.8677747941555003*(1+(\x)^2)/(1-(\x)^2)},{1.8677747941555003*2*(\x)/(1-(\x)^2)});
\draw [samples=50,domain=-0.99:0.99,rotate around={-25.177912521427697:(-2.2,-1.59)},xshift=-2.2cm,yshift=-1.59cm,line width = 0.5pt] plot ({1.8677747941555003*(-1-(\x)^2)/(1-(\x)^2)},{1.8677747941555003*(-2)*(\x)/(1-(\x)^2)});

\draw [samples=50,domain=-0.99:0.99,rotate around={-139.11320735863677:(-2.2,-1.59)},xshift=-2.2cm,yshift=-1.59cm,line width=0.5pt] plot ({1.0302153631127304*(1+(\x)^2)/(1-(\x)^2)},{1.0302153631127304*2*(\x)/(1-(\x)^2)});
\draw [samples=50,domain=-0.99:0.99,rotate around={-139.11320735863677:(-2.2,-1.59)},xshift=-2.2cm,yshift=-1.59cm,line width=0.5pt] plot ({1.0302153631127304*(-1-(\x)^2)/(1-(\x)^2)},{1.0302153631127304*(-2)*(\x)/(1-(\x)^2)});

\draw [samples=50,domain=-0.99:0.99,rotate around={-19.871215738808857:(-2.2,-1.59)},xshift=-2.2cm,yshift=-1.59cm,line width = 1.1pt, red] plot ({2.0672813413502626*(1+(\x)^2)/(1-(\x)^2)},{2.0672813413502626*2*(\x)/(1-(\x)^2)});
\draw [samples=50,domain=-0.99:0.99,rotate around={-19.871215738808857:(-2.2,-1.59)},xshift=-2.2cm,yshift=-1.59cm,line width = 1.1pt, red] plot ({2.0672813413502626*(-1-(\x)^2)/(1-(\x)^2)},{2.0672813413502626*(-2)*(\x)/(1-(\x)^2)});

\draw [line width = 0.5pt,domain=-8.325829998854992:7.364513866782811] plot(\x,{(--2.4401357007012363--0.8515182916907675*\x)/1.815422136650592});

\draw [fill=black] (-3.98,2.08) circle (1pt);
\draw[color=black, xshift = -10] (-3.872419569504368,2.379341531313904) node {$A$};
\draw [fill=black] (-4.6,-1.64) circle (1.5pt);
\draw[color=black, yshift = -15](-4.492863746968077,-1.3433235334683307) node {$B$};
\draw [fill=black] (0.2,-1.54) circle (1.5pt);
\draw[color=black, yshift = -15, xshift = -4] (0.3052378920846078,-1.2468099947517541) node {$C$};
\draw [fill=black] (-1.7701037227306569,2.652151402568374) circle (1pt);
\draw[color=black, xshift  = 5] (-1.6663958274111794,2.9446351152252803) node {$D$};
\draw [fill=black] (-2.2392459002745584,0.29380321317881286) circle (1pt);
\draw[color=black, xshift = -10, yshift = -3] (-2.135175872605982,0.5869472408631984) node {$T$};
\draw [fill=black] (-4.255005522010615,-1.3248809950025164) circle (1pt);
\draw[color=black, yshift = -12, xshift = 5] (-4.1068095921017695,-0.9848446753781894) node {$T_1$};
\draw [fill=black] (-1.5963425767365,-0.6365420223812464) circle (1pt);
\draw[color=black, yshift = -18, xshift = -8] (-1.4457934532018606,-0.29546225597407205) node {$T_2$};
\draw [fill=black] (-4.05466803692515,-0.5577150785119547) circle (1.5pt);
\draw[color=black, xshift = -8] (-3.899994866280533,-0.22652401403366026) node {$T_3$};
\draw [fill=black] (-2.2305574912891983,-0.12324041811846673) circle (1pt);
\draw[color=black, xshift = -10, yshift = -8] (-2.1213882242178994,0.1733177892207279) node {$O$};
\end{tikzpicture}
    \caption{$TT_3$ goes through the point $\infty_D$.}
    \label{fig3}
\end{figure}

    By a similar argument, conjugating the line $k$ with respect to $\triangle ABC$ yields that the points $A$, $B$, $C$, $D$, $T_1$, and $Q_D$ lie on a single conic. Therefore, to establish that $A$, $B$, $C$, $D$, $Q_A$, and $Q_D$ lie on a single conic, it suffices to prove that points $A$, $B$, $C$, $D$, $T_1$, and $T_2$ lie on a single conic.

    Now let $T_3$ denote the isogonal conjugate of $T_2$ with respect to $\triangle ABC$. We will show now that line $TT_3$ goes through point $\infty_D$ (Fig. \ref{fig3}). 

    Now, from Lemma~\ref{th3} it follows that $T_2$ moves along a conic passing through $D$, $B$, and $C$, as the isogonal conjugate of $T \in m$ with respect to $\triangle DBC$. Let us denote this conic by $\Gamma_1$. We then have the following maps:
    \begin{equation}
            T(m) \overset{}\longmapsto  T_2(\Gamma_1),
    \end{equation}
    \begin{equation}\label{eq3}
        T_2(\Gamma_1) \overset{}\longmapsto CT_2( \mathcal{L}(C)) \overset{}\longmapsto CT_3(\mathcal{L}(C)),
    \end{equation}
    \begin{equation}\label{eq4}
        T_2(\Gamma_1) \overset{}\longmapsto BT_2(\mathcal{L}(B)) \overset{}\longmapsto BT_3(\mathcal{L}(B)).
    \end{equation}
    From equations \eqref{eq3} and \eqref{eq4}, we obtain the following projective map: 
    \begin{equation}
        CT_3(\mathcal{L}(C))\overset{}\longmapsto BT_3(\mathcal{L}(B)).
    \end{equation}
    By Lemma \ref{th2}, we conclude that $T_3$ lies on a conic passing through points $B$ and $C$, which we denote by $\Gamma_2$. Thus, we have the following projective map:
    \begin{equation}
        BT_3(\mathcal{L}(B))\overset{}\longmapsto T_3(\Gamma_2).
    \end{equation}
    Note that since $D \in \Gamma_1$, it follows that $\infty_D \in \Gamma_2$. Therefore, we have the following projective map:
    \begin{equation}
        T_3(\Gamma_2) \overset{}\longmapsto \infty_DT_3(\mathcal{L}(\infty_D)) \overset{}\longmapsto m\cap\infty_DT_3 (m).
    \end{equation}
    By composing the above projective maps, we find that:
    \begin{equation}\label{2.13}
        T(m)\overset{}\longmapsto m\cap\infty_DT_3 (m).
    \end{equation}
   We will show that this map is the identity mapping. It is enough to show this in three cases, as a projective map is a Möbius transformation.

   \begin{enumerate}
       \item $T=\infty_m \implies T_2\in \Omega \implies T_3\in \ell_\infty \implies$ $T$, $T_3$, $\infty_D$ are collinear. Therefore, the map \eqref{2.13} is the identity mapping.
       \item Let $T = O$. By Lemma~\ref{th3}, we obtain that the conic passing through $A$, $B$, $C$, $H_D$, and $H_A$ is the isogonal conjugate of $TT_3$ with respect to $\triangle ABC$. Furthermore, by Proposition \ref{orthocenters}, we know that $A$, $B$, $C$, $H_D$, $H_A$, and $D$ lie on a single conic. Hence $\infty_D \in TT_3$, which proves the collinearity of $T$, $T_3$, and $\infty_D$. Therefore, the map \eqref{2.13} is the identity mapping.
       \item $T=M \implies T_2=D\implies T_3=\infty_D$. Thus $\infty_DT_3$ is tangent to $\Gamma_2$. Therefore by Lemma \ref{lemma1} (proved below) we obtain that this tangent line $DT_3$ passes through point $M=T$. Therefore, the map \eqref{2.13} is the identity mapping.
   \end{enumerate}
   
  We have shown that map \eqref{2.13} is the identity mapping in three cases. Therefore, it is the identity mapping in general. It follows directly that $T$, $T_3$, and $\infty_D$ are collinear.

       Now we take the isogonal conjugates of points on the line $TT_3$ with respect to $\triangle ABC$. Then, by Lemma~\ref{th3} and by taking the isogonal conjugates of the points $T$, $T_3$, and $\infty_D$ with respect to $\triangle ABC$, we obtain that points
    \begin{equation}
        A,\; B,\; C,\; D,\; T_1,\; T_2
    \end{equation}
    lie on a single conic. Which completes our proof.
    \end{proof}

\begin{figure}[ht]
    \centering
    \begin{tikzpicture}[line cap=round,line join=round,>=triangle 45,x=1cm,y=1cm]
\clip(-7.075834900193519,-5.387603053846561) rectangle (5.285589465798158,3.219155661097436);
\draw [fill = LightBlue, opacity = 1] (-3.98,2.08) -- (-4.6,-1.64) -- (0.25631585256672956,-1.6016388097362984) -- (-0.9305654105238774,2.427392492583314) -- cycle;
\draw [line width=0.5pt, fill = orange, opacity = 0.2] (-2.1836099969503824,-0.13106500050826925) circle (2.8488287908418153cm);
\draw [line width=0.5pt,domain=-7.075834900193519:7.285589465798158] plot(\x,{(-7.787896522996408--0.03836119026370155*\x)/4.85631585256673});
\draw [line width=0.5pt,domain=-2:2] plot(\x,{(--10.609327653434741--4.85631585256673*\x)/-0.03836119026370155});

\draw [samples=50,domain=-0.99:0.99,rotate around={32.17927914312907:(-2.171842073716635,-1.6208194048681488)},xshift=-2.171842073716635cm,yshift=-1.6208194048681488cm,line width=1.1pt, red] plot ({1.623334918027684*(1+(\x)^2)/(1-(\x)^2)},{1.623334918027684*2*(\x)/(1-(\x)^2)});
\draw [samples=50,domain=-0.99:0.99,rotate around={32.17927914312907:(-2.171842073716635,-1.6208194048681488)},xshift=-2.171842073716635cm,yshift=-1.6208194048681488cm,line width=1.1pt, red] plot ({1.623334918027684*(-1-(\x)^2)/(1-(\x)^2)},{1.623334918027684*(-2)*(\x)/(1-(\x)^2)});

\draw [samples=50,domain=-0.99:0.99,rotate around={-11.954257281119322:(-2.171842073716635,-1.6208194048681484)},xshift=-2.171842073716635cm,yshift=-1.6208194048681484cm,line width=1.1pt, red] plot ({2.3134292057315515*(1+(\x)^2)/(1-(\x)^2)},{2.3134292057315515*2*(\x)/(1-(\x)^2)});
\draw [samples=50,domain=-0.99:0.99,rotate around={-11.954257281119322:(-2.171842073716635,-1.6208194048681484)},xshift=-2.171842073716635cm,yshift=-1.6208194048681484cm,line width=1.1pt, red] plot ({2.3134292057315515*(-1-(\x)^2)/(1-(\x)^2)},{2.3134292057315515*(-2)*(\x)/(1-(\x)^2)});

\draw [line width=0.5pt] (-4.7787902820291,1.044037023772167)-- (0.3926827798587502,1.0848877155503465);
\draw [line width=0.5pt] (0.3926827798587502,1.0848877155503465)-- (0.43510613459583064,-4.285675833508466);
\draw [line width=0.5pt] (-0.9305654105238774,2.427392492583314)-- (-3.436654583376887,-2.6895224935998523);
\draw [line width=0.5pt] (-4.7787902820291,1.044037023772167)-- (0.4115702881283355,-1.3061670247887072);
\draw [line width=0.5pt,domain=-2.075834900193519:2.285589465798158] plot(\x,{(-2.715497100724024-2.979508808719759*\x)/0.023535846467493737});
\draw [line width=0.5pt] (-0.9070295640563837,-0.5521163161364452)-- (-3.436654583376887,-2.6895224935998523);
\draw [line width=0.5pt] (-4.7787902820291,1.044037023772167)-- (0.43510613459583064,-4.285675833508466);
\draw [fill=black] (-3.98,2.08) circle (1pt);
\draw[color=black, xshift = -3] (-3.8830050543271795,2.3546938253589262) node {$A$};
\draw [fill=black] (-4.6,-1.64) circle (1.5pt);
\draw[color=black, yshift = -16, xshift = -1] (-4.501379214198526,-1.368171014682832) node {$B$};
\draw [fill=black] (0.25631585256672956,-1.6016388097362984) circle (1.5pt);
\draw[color=black] (0.281555614194133,-1.8477264856034654) node {$C$};
\draw [fill=black] (-0.9305654105238774,2.427392492583314) circle (1.5pt);
\draw[color=black, xshift = 3] (-0.8289938974115504,2.69543060732885) node {$D$};
\draw [fill=black] (-2.1836099969503824,-0.13106500050826925) circle (1pt);
\draw[color=black, xshift = -5] (-2.0783620979679442,0.14621468296127302) node {$O$};
\draw [fill=black] (-0.9070295640563837,-0.5521163161364452) circle (1.5pt);
\draw[color=black, xshift = 3] (-0.8037541357841484,-0.28286126470455675) node {$H$};
\draw [fill=black] (-3.436654583376887,-2.6895224935998523) circle (1.5pt);
\draw[color=black, yshift = -14, xshift = -12] (-3.289870656083235,-2.415621122220005) node {$D^{\prime}$};
\draw [fill=red] (-2.171842073716635,-1.620819404868149) circle (1.5pt);
\draw[color=red, yshift = -20] (-1.9900229322720375,-1.3176914914280287) node {$M_{BC}$};
\draw [fill=black] (-4.7787902820291,1.044037023772167) circle (1.5pt);
\draw[color=black, xshift = -12, yshift = -5] (-4.678057545590339,1.3198635986354543) node {$K$};
\draw [fill=black] (0.3926827798587502,1.0848877155503465) circle (1pt);
\draw[color=black] (0.5465731112818528,1.357723241076557) node {$K^{\prime}$};
\draw [fill=black] (0.43510613459583064,-4.285675833508466) circle (1.5pt);
\draw[color=black, xshift = 4] (0.6349122769777594,-4.018345985560016) node {$K^{\prime \prime}$};
\draw [fill=black] (0.4115702881283355,-1.3061670247887072) circle (1.5pt);
\draw[color=black, xshift = 7, yshift = -3] (0.50871346884075,-1.0400541135266095) node {$L$};
\draw[color=uuuuuu, yshift = -85, xshift = -70] (0.06070091409393856,-0.1530049645584952) node {$\Omega$};
\draw[color=uuuuuu, yshift = -80, xshift = 100] (0.06070091409393856,-0.1530049645584952) node {$\mathcal{H}_1$};
\draw[color=uuuuuu, yshift = -140, xshift = 30] (0.06070091409393856,-0.1530049645584952) node {$\mathcal{H}_2$};
\end{tikzpicture}
    \caption{$M_{BC}$ is the center of the hyperbolas.}
    \label{fig:lemmaCenter}
\end{figure}
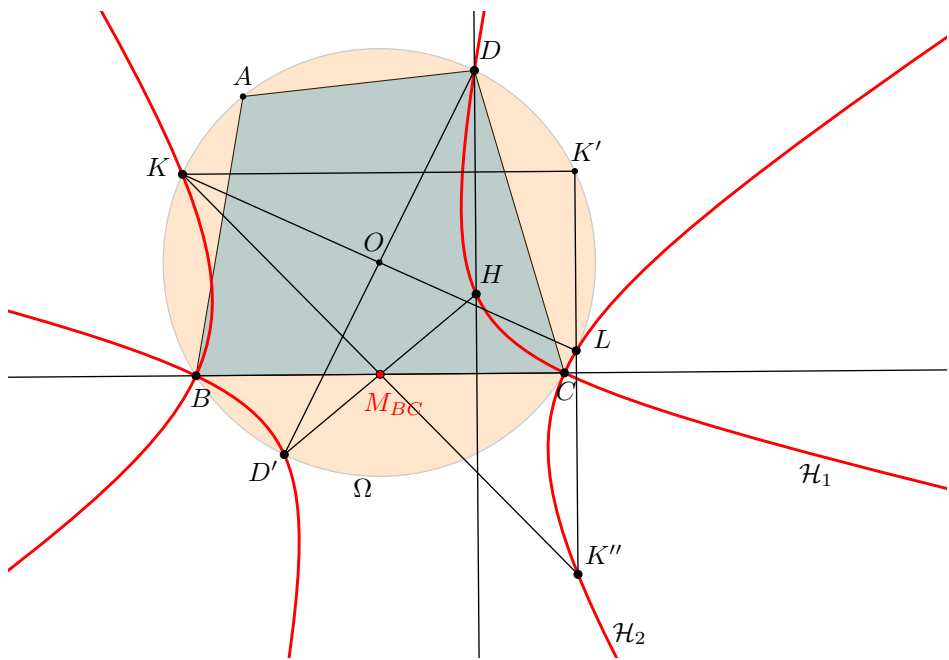

 \begin{lem}\label{lemma1}
       Let $A, B, C, D$ be four distinct points on a circle $\Omega$. Denote $\ell$ to be the perpendicular bisector of $BC$. Denote $\mathcal{H}_1$ to be the isogonal conjugate of $\ell$ with respect to $\triangle BCD$ and $\mathcal{H}_2$ to be the isogonal conjugate of $\mathcal{H}_1$ in with respect to $\triangle ABC$. Then, $\mathcal{H}_1$ and $\mathcal{H}_2$ are rectangular hyperbolas, and the midpoint of $BC$ is their center (Fig. \ref{fig:lemmaCenter}).
   \end{lem}

\begin{rem}
    To finish the proof of Theorem \ref{themG}, we only really need the part that $M_{BC}$ is the center of $\mathcal{H}_2$, but for completeness, we include the other result too.
\end{rem}

   \begin{proof}
       First, it is clear that $\mathcal{H}_1$ is a rectangular hyperbola -- the intersections of $\ell$ with $\Omega$ are antipodal on $\Omega$, hence it follows that the asymptotes of $\mathcal{H}_1$ are perpendicular. Let $\mathcal{H}_1$ intersect $\Omega$ for the fourth time at $D'$, and let $H$ be the orthocenter of $\triangle BCD$. Notice that $D'$ is the antipode of $D$ on $\Omega$, as $\infty_{D'}$ -- the isogonal conjugate of $D'$ with respect to $\triangle BCD$ is the point at infinity of $\ell$, and as $\ell \perp BC \implies D\infty_D \perp BC \implies H \in D\infty_D $, we must have $DD'$ passing through the center of $\Omega$ (in the exceptional case where $D \in \ell$, its isogonal conjugate will be a line. For our lemma to make sense, in such case we assume that the isogonal conjugate of $D$ with respect to $\triangle BCD$ is the whole line $BC$, and $\mathcal{H}_1$ degenerates to two perpendicular lines -- namely $BC$ and $\ell$). From Lemma \ref{th4}, the center of $\mathcal{H}_1$ is the midpoint of $HD'$, but it is clearly also the midpoint of $BC$. 

       Now, as $\mathcal{H}_1$ is a conic passing through $B$ and $C$, $\mathcal{H}_2$, being its isogonal conjugate with respect to $\triangle ABC$, must also be a conic passing through $B$ and $C$ (in the exceptional case where $A \in \mathcal{H}_1$, we again assume that $\mathcal{H}_2$ degenerates to two perpendicular lines -- the same as before). Isogonal conjugates of $D$ and $D'$ with respect to $\triangle ABC$ lie on $\mathcal{H}_2$ and are at infinity. Because $AD \perp AD'$ (as $DD'$ is the diameter of $\Omega$), we get that $\mathcal{H}_2$ is also a rectangular hyperbola. Now label the points at infinity of $\mathcal{H}_1$ as $\infty_K$ and $\infty_L$, and let $K$ and $L$ be their isogonal conjugates with respect to $\triangle ABC$, respectively. It is clear that $K$ and $L$ lie on $\Omega$ as well as on $\mathcal{H}_2$. $K$ and $L$ are also the antipodes of each other on $\Omega$, because 
       \begin{equation}
            \arcangle KAL = \arcangle \infty_KA\infty_L = 90^{\circ}.
       \end{equation}
       Let $K'$ be the reflection of $K$ in $\ell$. Now $KK' \parallel BC$ and $\arcangle KK'L = 90^{\circ} \implies K'L \perp BC$. Let $K''$ be the reflection of $K'$ in $BC$. It follows that $K''$ is the orthocenter of $\triangle BCL$, and from Lemma \ref{th4}, the midpoint of $KK''$ is the center of $\mathcal{H}_2$. As $BC$ passes through the midpoint of $K'K''$ and $BC \parallel KK'$, we get that the midpoint of $KK''$ lies on $BC$.
   \end{proof}

    A more synthetic proof of the main result can be seen below.

    \begin{proof}

     Consider the isogonal conjugate of the circumconic $ABCDQ_D$ with respect to $\triangle ABC$. It will be a line, denoted $k$, through $P_D$ and $\infty_D$, where $\infty_D$ is the isogonal conjugate of $D$ with respect to $\triangle ABC$. Similarly define the line $\ell$ through $P_A$ and $\infty_A$, where $\infty_A$ is the isogonal conjugate of $A$ with respect to $\triangle BCD$. 

     Repeating the same treatment as at the start of the first proof of Theorem \ref{themG}, we get that $k$ and $\ell$ are reflections of each other in the perpendicular bisector of $BC$.

    \begin{figure}[ht]
    \centering
    \begin{tikzpicture}[line cap=round,line join=round,>=triangle 45,x=1cm,y=1cm]
\clip(-6.034414809464163,-3.982528704228925) rectangle (6.020360158233172,4.23827496267534);
\draw [fill = LightBlue, opacity = 1] (-2.714896924246057,-1.6215223374099108) -- (-1.8624646656185977,2.555626218624744)  -- (1,3) -- (2.7148969242460566,-1.6215223374099108) -- cycle;
\draw [line width = 0.5pt, fill = orange, opacity = 0.2] (0,0) circle (3.1622776601683795cm);
\draw [line width = 0.5pt] (-1.8624646656185977,2.555626218624744)-- (1,3);
\draw [line width = 0.5pt] (1,3)-- (2.7148969242460566,-1.6215223374099108);
\draw [line width = 0.5pt] (2.7148969242460566,-1.6215223374099108)-- (-2.714896924246057,-1.6215223374099108);
\draw [line width = 0.5pt] (-2.714896924246057,-1.6215223374099108)-- (-1.8624646656185977,2.555626218624744);
\draw [line width = 0.5pt,domain=-6.034414809464163:0] plot(\x,{(-0-0.6874184561950774*\x)/-1.8624646656185972});
\draw [line width = 0.5pt,domain=0:6.020360158233172] plot(\x,{(-0-0.24304467481982112*\x)/1});
\draw [line width = 0.5pt] (0,-3.982528704228925) -- (0,4.23827496267534);
\draw [line width = 0.5pt,domain=-6.034414809464163:6.020360158233172] plot(\x,{(--3.0317677782310493--0.4443737813752562*\x)/0.8624646656185977});
\draw [line width = 0.5pt,domain=-6.034414809464163:6.020360158233172] plot(\x,{(--3.031767778231049-0.4443737813752566*\x)/0.8624646656185975});

\draw [line width=1pt,domain=-6.034414809464163:6.020360158233172, red] plot(\x,{(--0.593159607754542--0.5066219162802141*\x)/0.9832792120799909});
\draw [line width=1pt,domain=-6.034414809464163:6.020360158233172, red] plot(\x,{(--0.520278672095901-0.4443737813752566*\x)/0.8624646656185975});

\draw [fill=black] (-2.714896924246057,-1.6215223374099108) circle(1pt);
\draw[color=black, yshift = -15, xshift = -10] (-2.6210986670653167,-1.3684573627791037) node {$B$};
\draw [fill=black] (-1.8624646656185977,2.555626218624744) circle(1pt);
\draw[color=black, xshift = -5] (-1.7677696314656055,2.81405678354061) node {$A$};
\draw [fill=black] (1,3) circle(1pt);
\draw[color=black] (1.0926854456151178,3.258749379557361) node {$D$};
\draw [fill=black] (2.7148969242460566,-1.6215223374099108) circle(1pt);
\draw[color=black, yshift = -10, xshift = 2] (2.811362235625805,-1.3684573627791037) node {$C$};
\draw [fill=black] (0,0) circle(1pt);
\draw[color=black, xshift = -8] (0.09513178428024374,0.2540696767414749) node {$O$};
\draw [fill=black] (-1.8624646656185972,-0.6874184561950774) circle(1pt);
\draw[color=black, yshift = -15] (-1.7317134750318148,-0.39494113906675654) node {$H_D$};
\draw [fill=black] (1,-0.24304467481982112) circle(1pt);
\draw[color=black, yshift = -5] (0.9364421010686919,-0.38292242025549295) node {$H_A$};
\draw [fill=black] (2.2162497436861823,-0.5386476982737202) circle(1.5pt);
\draw[color=black] (1.9700519188373566,-0.6473342341032909) node {$P_A$};
\draw [fill=black] (-4.127686837801787,-1.5234909773474912) circle(1.5pt);
\draw[color=black] (-4.339775457076004,-1.0920268301200422) node {$P_D$};
\draw [fill=black] (-1,3) circle(1pt);
\draw[color=black] (-0.8543470018095762,3.258749379557361) node {$D'$};
\draw [fill=black] (1.8624646656185975,2.5556262186247434) circle(1pt);
\draw[color=black] (2.0061080752711473,2.81405678354061) node {$A'$};
\draw[color=black, xshift =15, yshift = -5] (-5.890190183729001,-2.1977489607562886) node {$k$};
\draw[color=black,yshift = -4] (-5.890190183729001,3.5952735062727403) node {$\ell$};
\draw [fill=red] (3.015985308264658,-0.9507011204010112) circle(1.5pt);
\draw[color=red, yshift = -20] (3.1478863623411835,-0.6593529529145545) node {$L_1$};
\draw [fill=red] (-2.524760576000225,1.9040966450984083) circle(1.5pt);
\draw[color=red, yshift = -20, xshift = -5] (-2.39274300965131,2.201102124166169) node {$L_2$};
\draw [fill=red] (2.5247605760002267,1.9040966450984054) circle(1.5pt);
\draw[color=red, yshift = -18] (2.655118891079378,2.201102124166169) node {$K_2$};
\draw [fill=red] (-3.0159853082646575,-0.9507011204010115) circle(1.5pt);
\draw[color=red] (-3.3782779521749204,-0.6593529529145545) node {$K_1$};
\end{tikzpicture}
    \caption{$k$ and $\ell$ are reflections of each other in the perpendicular bisector of $BC$. }
    \label{kazik2}
\end{figure}
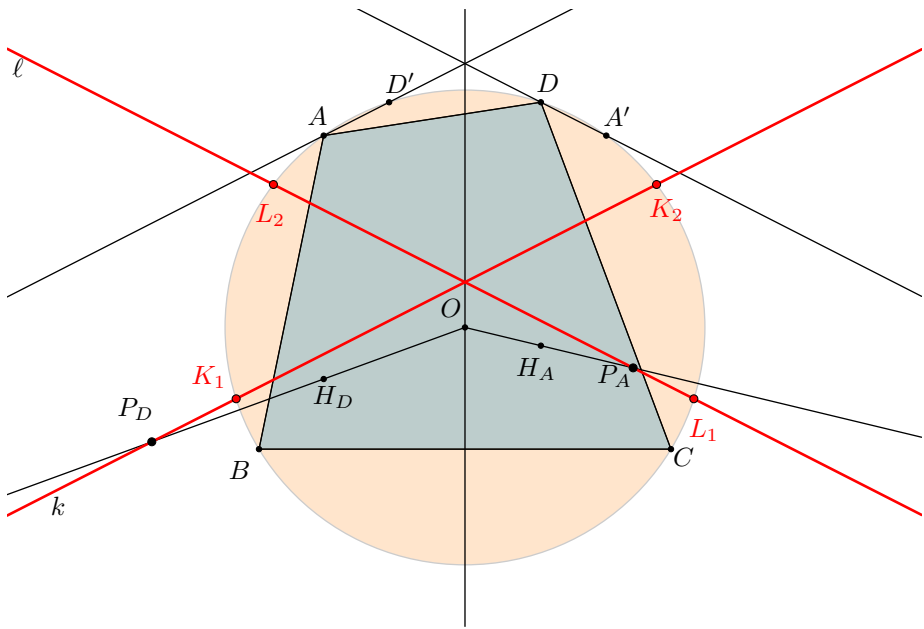

    We have to prove that the isogonal conjugate of $\ell$ with respect to $\triangle BCD$ and the isogonal conjugate of $k$ with respect to $\triangle ABC$ are the same conic. To do this, we first assume that the conics are hyperbolas, that is, lines $k$ and $\ell$ intersect the circle on $ABCD$, each at two points. 

    As the aforementioned isogonal conjugates of $k$ and $\ell$ share points $A,B,C,D$, it is sufficient to prove that they have the same points at infinity. Let $k$ intersect $\Omega$ at $K_1$ and $K_2$, and let $\ell$ intersect $\Omega$ at $L_1$ and $L_2$ (Fig. \ref{kazik2}).

     As $k$ and $\ell$ are reflections of each other in the perpendicular bisector of $BC$, we can see that $K_1K_2L_1L_2$ is a trapezoid with its bases parallel to $BC$. Without loss of generality, assume that $K_1$ and $L_2$ lie on the same side of the aforementioned bisector. Now because arcs $A'K_2$ and $AL_2$ have the same length and $L_1L_2 \parallel DA'$, we get that $DK_2 \parallel AL_1$. Now from the trapezoid, the line $DK_2$ is clearly the isogonal conjugate of the line $DL_2$ in the angle $BDC$, and $AL_1$ is the isogonal conjugate of the line $AK_1$ in the angle $BAC$, hence, from $DK_2 \parallel AL_1$, the isogonal conjugate of $L_2$ in $BCD$ is the same point at infinity as the isogonal conjugate of $K_1$ in $ABC$, which proves our claim.

    Now there are clearly infinitely many instances of $P_A$, so that $k$ and $\ell$ have intersections with $\Omega$ (when $P_A$ is inside the circle), hence, we have proven our Theorem for infinitely many cases. To get from this to the Theorem for all the cases, we can use a \emph{polynomial moving points} argument. When $P_A$ moves on $OH_A$, $Q_A$, being its isogonal conjugate, by Lemma \ref{th2}, moves on a conic $\Gamma_1$, and the map 
    \begin{equation}
       P_A(OH_A) \longmapsto Q_A(\Gamma_1) 
    \end{equation}
     is projective. Similarly $Q_D$ also moves on a conic $\Gamma_2$, and the map 
     \begin{equation}
         P_A(OH_A) \longmapsto P_D(OH_D) \longmapsto Q_D(\Gamma_2)
     \end{equation}
      is also projective. Now we get that the map
    \begin{equation}
    Q_A(\Gamma_1) \longmapsto P_A(OH_A) \longmapsto P_D(OH_D) \longmapsto Q_D(\Gamma_2)
    \end{equation}
     is projective, since projective maps are bijective. Therefore we can express coordinates of the point $Q_D$, as a polynomial of coordinates of the point $Q_A$, which lies on a fixed conic. Thus, to prove that $A,B,C,D,Q_A,Q_D$ lie on a single conic it is sufficient to check some finite number of cases, but we have infinitely many.   \end{proof}

\subsection{A more algebraic approach}

As it turns out this result also comes about purely projectively, using techniques well known in algebraic geometry. To proceed with the proof, consider the space of all conics considered as homogenous forms of degree $2$. Conics passing through any four non-collinear points $P_1$, $P_2$, $P_3$, $P_4$ form a line in that space; let us call this line $C(P_1,P_2,P_3,P_4) \cong \mathbb{P}^1$. We will need the following Lemma.

\begin{lem} \label{lemks}
        		Let $ABC$ be a triangle, $\ell$ be a line different from the line in infinity that passes through at most one vertex of the triangle, $X$ be a point in the plane not lying on any of the triangle's sides such that its isogonal conjugate does not lie on $\ell$. Then there exists a projectivity from $\ell$ to $C(A,B,C,X)$ that maps each point possessing an isogonal conjugate to a conic passing through the isogonal conjugate, which moreover satisfies the two following conditions:
                \begin{enumerate}
                    \item if $X$ lies on the circumcircle of $ABC$, the image of the point in infinity is the circumcircle of $ABC$,
                    \item if $\ell$ passes through the circumcenter $O$ of $\triangle ABC$ then the image of $O$ is a rectangular hyperbola (perhaps one reducing to two perpendicular lines).
                \end{enumerate}
        \end{lem}
        \begin{proof}
        		First consider the case when $\ell$ does not pass through any of the triangle's vertices. 
                Take the conic $\Gamma$ and the parametrisation $\gamma: \ell \rightarrow \Gamma$ from Lemma \ref{th2}; had $X$ lied on $\Gamma$, it would be the image of some point $X'$ in $\gamma$. It is easy to see $X'$ must not lie on any of the triangle's sides, meaning that $X$ is the isogonal conjugate of $X'$, a contradiction. 
                The parametrisation $\gamma$ induces a map in the reverse direction: from the space of homogenous forms of degree $2$ on $\mathbb{P}^2$ restricted to $\Gamma$ to the space $H^4(2)$ of homogenous forms of degree $4$ on $\ell$. In particular, we have a map $\phi$ from $C(A, B, C, X)$ to $H^4(2)$. $0$ does not lie in the image of $\phi$, as $X$ does not lie on $\Gamma$ -- hence $\phi$ is a well-defined projectivity from $C(A, B, C, X)$ to $\text{Im } \phi$. Now, three of the roots of any form in the image of $\phi$ correspond to the preimages of $A$, $B$, $C$ in the aforementioned parametrisation; we can compute the fourth root by Viete's formulas, inducing a projective map $\psi$ from $\text{Im } \phi$ to $\ell$. Both $\psi$ and $\phi$ are projectivities -- denote the inverse of their composition by $f$. By our construction, it is a map from $\ell$ to $C(A, B, C, X)$ that sends every point to the conic passing through its image in $\gamma$; thus $f$ is the sought after projectivity. 
                
        		In the case where $\ell$ passes through one of the triangle's vertices (say $A$), the reasoning is very similiar except we directly restrict conics to the line $\Gamma$ given by Lemma 2.2. We know one of the resulting form's roots will be $A$, and we extract the second one by Vieta's formulas as above. 
                
                The last two claims are easy to prove synthetically except two nuances; the first proposition has a slight problem when $\ell$ is parallel to one of the triangle's sides, and the second -- when $ABC$ is right-angled. These issues can be addressed manually or with an argument by continuity.   
        \end{proof}

    Now we can give another proof of the main result of this paper.
    \begin{proof}
        We will assume the isogonal conjugate of $A$ with respect to $\triangle BCD$ does not lie on $OH_A$ and similarly for $D$; otherwise, it is not hard to see the conic degenerates to a rectangular hyperbola through $A$, $B$, $C$, $D$. 
        
        By $\infty_{H_A H_D}$ denote the point at infinity on line $H_A H_D$. We will use our previous Lemma for the line $OH_A$, the triangle $BCD$ and the point $A$, as well as the line $OH_D$, the triangle $ABC$ and the point $D$; call the obtained projectivities $\phi_A$ and $\phi_D$.  Consider the following diagram:
        \[\begin{tikzcd}
OH_A \arrow[r, "\phi_A"] \arrow[d, "\pi_{\infty_{H_A H_D}}"] & C(A,B,C,D) \\
\mathcal{L}(\infty_{H_A H_D}) \arrow[r, "\pi_{OH_D}"] & OH_D \arrow[u, "\phi_D"]
\end{tikzcd}\]
where $\pi_{\infty_{H_A H_D}}$ is a map which assigns each point $X$ on $OH_A$ the line $\infty_{H_A H_D}$ and $\pi_{OH_D}$ is a map which assigns each line through $\infty_{H_A H_D}$ its intersection with $OH_D$. If the diagram commutes, then applying the relevant maps to $P_A$ will get us that the conic passing through its isogonal conjugate is the same as the conic passing through the isogonal conjugate of $P_D$, ending the proof. But all maps on the diagram are projectivities; hence it suffices to prove that $\phi_A(X)=\phi_D(\pi_{OH_D}(\pi_{\infty_{H_A H_D}}(X))$ for three cases of $X$. 
\begin{enumerate}
\item For $X$ at infinity, $\pi_{\infty_{H_A H_D}}(X)$ is at infinity as well; by Lemma \ref{lemks}, both conics are the circumcircle. %
\item For $X$ at $H_A$, $\phi_A(X)$ passes through $A$, $B$, $C$, $D$, $O$; there is exactly one such conic, and it must also be equal to $\phi_D(\pi_{\infty_{H_A H_D}}(X))$, ending the proof for this case.
\item For $X$ at $O$, $\phi_A(X)$ is a rectangular hyperbola by Lemma \ref{lemks}. The same follows for $\phi_D(\pi_{\infty_{H_A H_D}}(X))$; both conics pass through $A$, $B$, $C$, $D$, implying they are one and the same. \qedhere
\end{enumerate}
\end{proof}
    
\section{Connection to the Encyclopedia of Triangle Centers}

\subsection{Centers with constant Shinagawa coefficients}

Theorem~\ref{themG} holds for any point $P_D$ on the Euler line, which in turn establishes a ratio $P_{D}H_{D} : P_{D}O = \lambda_0$. Since this condition is satisfied by triangle centers such as the orthocenter, it is natural to investigate the cases where $P_{A}, P_{B}, P_{C}, P_{D}$ correspond to other triangle centers. To this end, we consult the Encyclopedia of Triangle Centers (henceforth ETC)~\cite{kimberling}.

Let $X_n$ denote the $n$-th Kimberling center cataloged in the ETC. We consider the subset of these centers that lie on the Euler line. The homogeneous barycentric coordinates of such points, relative to the affine frame defined by triangle $ABC$, are conveniently expressed using Shinagawa coefficients~\cite{kimberling}. We use the convention where $A = (1:0:0), B=(0:1:0), C=(0:0:1)$, and the Conway symbols for $\triangle ABC$ with side lengths $a,b,c$ are given by $S = 2 \cdot \text{Area}(\triangle ABC)$ and $S_A = \frac{1}{2}(b^2+c^2-a^2)$, with $S_B$ and $S_C$ defined cyclically~\cite{yiu}.

\begin{defn}[Shinagawa coefficients]
Let $X_n$ be a triangle center on the Euler line. Its homogeneous barycentric coordinates are expressed as $(f(a,b,c):f(b,c,a):f(c,a,b))$, where $f$ is the triangle center function of $X_n$~\cite{kimberling}. The \textit{Shinagawa coefficients} of $X_n$ is the pair of functions $(G(a,b,c), H(a,b,c))$ that satisfy the relation:
\begin{equation} \label{eq:shinagawa}
f(a,b,c) = G(a,b,c) \cdot S^2 + H(a,b,c) \cdot S_BS_C,
\end{equation} and analogous cyclic conditions hold.
\end{defn}

\begin{rem}
    Note that Shinagawa coefficients are homogeneous.
\end{rem}

It is often the case that the Shinagawa coefficients $(G(a,b,c), H(a,b,c))$ are constants for all possible $a, b, c$. For example, the coefficients of the orthocenter, $X_4$, are $(0,1)$ (see \cite{shinagawa} for more examples).

In this context, we claim that our points $P_A, P_B, P_C, P_D$ are always centers with constant Shinagawa coefficients. To establish this, we first present the following lemma from \cite{wyrzyk1}, where lowercase letters of points denote their positional vectors.

\begin{lem}\label{conversion}
    Let $ABC$ be a triangle on the unit circle with center $O$ and orthocenter $H$. Let $X$ be a center with constant Shinagawa coefficients $(u,v)$. Then there exists a unique scalar $\lambda \in \mathbb{R} \cup \{ \infty \}$ such that $x = \lambda h$, where
    \begin{equation}
        \lambda = \frac{u+v}{3u+v}.
    \end{equation} 
\end{lem}

\begin{prop}\label{prop:shinagawa_equiv}
    The points $P_{A},\; P_{B},\; P_{C},\; P_{D}$ satisfy the condition in \eqref{eq2} if and only if they correspond to the same triangle center with constant Shinagawa coefficients in their respective triangles. 
\end{prop}

\begin{proof}
Assume the condition in \eqref{eq2} holds. This gives a constant $\lambda_0$ such that
\begin{equation}
    \frac{p_D-h_D}{p_D} = \lambda_0,
\end{equation}
which implies $p_D = \lambda h_D$ for $\lambda = (1-\lambda_0)^{-1}$. The condition ensures this same $\lambda$ applies to all four points, so $p_X = \lambda h_X$ for $X \in \{A,B,C,D\}$. By Lemma~\ref{conversion}, this single $\lambda$ determines a unique ratio of Shinagawa coefficients $(u,v)$:
\begin{equation}\label{eq:uv_ratio}
    \frac{u}{v} =  \frac{1-\lambda}{3\lambda - 1}.
\end{equation}
Since this ratio is the same for all four points, they must correspond to the same center with constant Shinagawa coefficients.

Conversely, assume $P_A, P_B, P_C, P_D$ are the same center with constant Shinagawa coefficients $(u,v)$. By Lemma~\ref{conversion}, these coefficients determine a unique scalar $\lambda$ via \eqref{eq:uv_ratio}. This implies $p_X = \lambda h_X$ for each point, which is equivalent to the ratio condition in \eqref{eq2}.
\end{proof}

 Following Proposition \ref{prop:shinagawa_equiv} the isogonal conjugates of centers with constant Shinagawa coefficents are excatly the family of points that satisfy the main result of this paper, meaning it can be rephrased as follows.

\begin{thm}\label{them1} Assume $X$ is a triangle center with constant Shinagawa coefficients $(u,v)$. Let $ABCD$ be a cyclic quadrilateral, by $X_A, X_B, X_C, X_D$ denote the $X$ points of $\triangle BCD$, $\triangle ACD$, $\triangle ABD$ and $\triangle ABC$, respectively. Furthermore by $Y_A$ denote the isogonal conjugate of $X_A$ with respect to $\triangle BCD$, similarly define $Y_{B},\; Y_{C},\; Y_{D}$. Then $A,\; B,\; C,\; D,\; Y_A,\; Y_{B},\; Y_{C},\; Y_{D}$ lie on a conic, denoted by $\Phi(u,v)$.
\end{thm}

\subsection{A catalog of centers that satisfy the main result}

We now identify all cataloged pairs of centers $X_n$ and their isogonal conjugates $X_n^{\prime}$ that satisfy Theorem \ref{them1}. See Table \ref{tab1}.

\begin{table}[ht]
  \centering
  \begin{tabular}{@{}cc|cc|cc@{}}
      $X_n$ & $X_n^{\prime}$ & $X_n$ & $X_n^{\prime}$ & $X_n$ & $X_n^{\prime}$ \\ \hline
      $X_2$ & $X_6$ & $X_3$ & $X_4$ & $X_{5}$ & $X_{54}$ \\ \hline
      $X_{20}$ & $X_{64}$ & $X_{30}$ & $X_{74}$ & $X_{381}$ & $X_{3431}$ \\ \hline
      $X_{140}$ & $X_{1173}$ & $X_{376}$ & $X_{3426}$ & $X_{547}$ & $X_{57714}$ \\ \hline
      $X_{382}$ & $X_{11270}$ & $X_{546}$ & $X_{57713}$ & $X_{550}$ & $X_{16835}$ \\ \hline
      $X_{548}$ & $X_{57715}$ & $X_{549}$ & $X_{14483}$ & $X_{1656}$ & $X_{13472}$ \\ \hline
      $X_{631}$ & $X_{3527}$ & $X_{632}$ & $X_{57730}$ & $X_{3091}$ & $X_{14528}$ \\ \hline
      $X_{1657}$ & $X_{13452}$ & $X_{3090}$ & $X_{43908}$ & $X_{3523}$ & $X_{52518}$ \\ \hline
      $X_{3146}$ & $X_{3532}$ & $X_{3522}$ & $X_{22334}$ & $X_{3534}$ & $X_{11738}$ \\ \hline
      $X_{3524}$ & $X_{3531}$ & $X_{3529}$ & $X_{43719}$ & $X_{3830}$ & $X_{20421}$ \\ \hline
      $X_{3543}$ & $X_{43713}$ & $X_{3628}$ & $X_{34567}$ & $X_{5071}$ & $X_{44731}$ \\ \hline
      $X_{5054}$ & $X_{14491}$ & $X_{5059}$ & $X_{43691}$ & $X_{12100}$ & $X_{14487}$ \\ \hline
      $X_{8703}$ & $X_{13603}$ & $X_{10304}$ & $X_{14490}$ & $X_{46853}$ & $X_{46851}$ \\ \hline
      $X_{33703}$ & $X_{44763}$ & $X_{33923}$ & $X_{46848}$ & $X_{61138}$ & $X_{61137}$
  \end{tabular}
  \caption{Pairs $(X_n, X_n^{\prime}$), where $X_n$ is a center with constant Shinagawa coefficients.}\label{tab1} 
\end{table}
\begin{rem}
Note that for $n \leq 61371$, there exist $721$ centers with constant Shinagawa coefficients; however, most of them do not have a cataloged isogonal conjugate pair \cite{kimberling, shinagawa, wyrzyk1} (as of 20th of July 2025).
\end{rem}

\subsection{Equation of the eight-point conic family} 

We conclude our study by providing the explicit equation for the conic $\Phi(u,v)$ from Theorem~\ref{them1}. We first recall two results concerning barycentric coordinates~\cite{yiu}.

\begin{lem}\label{lem:conic_eqn}
Let $D = (d:e:f)$ and $X = (p:q:r)$ be two distinct points, with homogeneous barycentric coordinates given relative to the affine frame defined by $\triangle ABC$. The conic passing through the five points $A, B, C, D, X$ has the equation:
\begin{equation}
    \sum_{\text{cyc}} frxy(ep - dq) = 0.
\end{equation}
\end{lem}

\begin{lem}[Isogonal conjugate coordinates]\label{lem:isogonal_coords}
Let a point $P$ have homogeneous barycentric coordinates $(x:y:z)$ relative to the affine frame defined by $\triangle ABC$. Its isogonal conjugate with respect to $\triangle ABC$ is the point $Q$ given by:
\begin{equation}
    Q = \left( \frac{a^2}{x} : \frac{b^2}{y} : \frac{c^2}{z} \right),
\end{equation}
where $a,b,c$ are the lengths of sides $BC, AC, AB$ respectively.
\end{lem}

\begin{prop} In the affine frame of $\triangle ABC$, the equation of the conic passing through $A, B, C, D, Y_A, Y_B, Y_C, Y_D$, denoted $\Phi(u,v)$ where $(u,v)$ are the Shinagawa coefficients of $X_i$, is given by
\begin{equation}\label{eq:conic_final}
    \sum_{\text{cyc}} \frac{c^2 r xy}{uS^2+vS_{A}S_{B}} \left( \frac{b^2 p}{uS^2+vS_{A}S_{C}} - \frac{a^2 q}{uS^2+vS_{B}S_{C}} \right) = 0,
\end{equation}
where $D=(p:q:r)$ are the homogeneous barycentric coordinates of the fourth vertex.
\end{prop}

\begin{proof}
By Theorem~\ref{themG}, all eight of our points lie on a single conic, which in turn is uniquely determined by the five points $A,B,C,D,$ and $Y_D$. To find its equation, we can, without loss of generality, work within the affine frame of $\triangle ABC$. 

By Proposition~\ref{prop:shinagawa_equiv}, $Y_D$ is the isogonal conjugate of the center $X_D$ with constant Shinagawa coefficients $(u,v)$. Applying Lemma~\ref{lem:isogonal_coords}, the coordinates of $Y_D$ in the frame of $\triangle ABC$ are:
\begin{equation}
    Y_D = \left( \frac{a^2}{uS^2+vS_BS_C} : \frac{b^2}{uS^2+vS_AS_C} : \frac{c^2}{uS^2+vS_AS_B}\right).
\end{equation}
Substituting the coordinates of $D=(p:q:r)$ and $Y_D$ into the five-point conic formula (Lemma~\ref{lem:conic_eqn}) directly yields the equation for $\Phi(u,v)$ as stated.
\end{proof}

\begin{rem}
A computational verification that all eight points lie on $\Phi(u,v)$ is also possible. The procedure involves performing an affine transformation~\cite{Gallier} to reframe the coordinates of each point $Y_i$ from its native triangle into the affine frame of $\triangle ABC$. One can then confirm by direct substitution that these transformed coordinates are zeros of Equation~\eqref{eq:conic_final}.
\end{rem}

\end{document}